\documentclass[11pt]{article}
\usepackage[utf8]{inputenc}
\usepackage[T1]{fontenc}

\usepackage{cite}
\usepackage{floatrow}
\usepackage{amsthm}
\usepackage{amsmath}
\usepackage{tikz}
\usepackage{mathrsfs} 
\usepackage{fullpage}
\usepackage{thmtools}
\usepackage{thm-restate}
\usepackage{mathtools}
\usepackage{amssymb}
\usepackage{lmodern, comment}
\usepackage{graphicx}
\usepackage{pgf,tikz,tkz-graph,subcaption}
\usetikzlibrary{arrows,shapes}
\usetikzlibrary{decorations.pathreplacing}
\usepackage{url}
\usepackage[capitalise]{cleveref}
\usepackage[makeroom]{cancel}
\usepackage[normalem]{ulem}

\newtheorem{theorem}{Theorem}

\newtheorem{corollary}[theorem]{Corollary}
\newtheorem{lemma}[theorem]{Lemma}

\newtheorem{proposition}[theorem]{Proposition}

\newcommand*{\abs}[1]{\lvert #1\rvert}

\begin{document}

\title{Benjamini--Schramm convergence and subtrees of trees}
\author{
 Stijn Cambie \thanks{Department of Computer Science, KU Leuven Campus Kulak-Kortrijk, 8500 Kortrijk, Belgium. 
 Supported by a postdoctoral fellowship by the Research Foundation Flanders (FWO) with grant number 1225224N. E-mail: {\tt stijn.cambie@hotmail.com}}
 \and 
Stephan Wagner\thanks{Institute of Discrete Mathematics, TU Graz, Austria and Department of Mathematics, Uppsala University, Uppsala, Sweden, supported by the Swedish research council (VR), grant  2022-04030. E-mail: {\tt stephan.wagner@tugraz.at}}
\and Ruoyu Wang
\thanks{Department of Mathematics, Uppsala University, Uppsala, Sweden, supported by the Swedish research council (VR), grant  2022-04030. E-mail: {\tt ruoyu.wang@math.uu.se}}
}		

\maketitle

\begin{abstract}
In this paper, we study the asymptotic behaviour of the number of subtrees and the subtree density for a sequence of trees that converges in the Benjamini--Schramm sense. Benjamini--Schramm convergence, also called local weak convergence, describes the local behaviour of a sequence of graphs. Here we show that for a  Benjamini--Schramm-convergent sequence of trees, the subtree entropy, i.e., the logarithm of the number of subtrees divided by the order, converges to a constant depending only on the limit. The same holds true for the subtree density, i.e., the probability of a uniformly random vertex being contained in a uniformly random subtree, provided that long paths are ruled out in the limit. Related to this, we show that the subtree density and the average subtree entropy are dense in different parts of the unit interval $[0,1]$ for both general trees and series-reduced trees.
\end{abstract}

\section{Introduction}

In this paper, we study the behaviour of two quantities associated with subtrees for a Benjamini--Schramm convergent sequence of finite trees. Let us start by defining the necessary notions.

All graphs we consider in this work are undirected, connected and locally finite (i.e., all vertices have finite degrees). A sequence of finite graphs $G_n$ is \emph{locally weakly convergent} if for any positive integer $r$ and finite rooted graph $H$, the probability $\sigma_n(H,r)$ of $H$ being isomorphic (as a rooted graph) to the (closed) $r$-ball centred at a uniformly random vertex of $G_n$ converges to a limit $\sigma(H,r)$ such that $\sigma(\cdot,r)$ is a probability measure for every $r$. Intuitively, local weak convergence of graphs means the convergence of sampling statistics in any finite ``radius of sight''.

Local weak convergence is also called \emph{Benjamini--Schramm convergence} after the influential paper of Benjamini and Schramm~\cite{BS01}. The asymptotic behaviour of various graph parameters under this notion of convergence has been studied. In particular, Lyons~\cite{Lyons05} showed that the tree entropy, i.e., the logarithm of the number of spanning trees divided by the order of the graph, converges for Benjamini--Schramm convergent sequences. Abért, Csikvári, Frenkel, and Kun \cite{ACFK16} showed an analogous result for the number of matchings instead of the number of spanning trees. A graph parameter $P$ with the property that $P(G_n)$ converges for all Benjamini--Schramm convergent sparse graph sequences $G_n$ in a certain class of graphs are also called \emph{estimable} \cite{ACFK16}. See \cite{ACH15,AH15,ATV13,CF16} for further examples, especially connected to roots of graph polynomials.  

We will see that similar results also hold for quantities associated with subtrees of trees. Here, a subtree is simply any nonempty subgraph of a tree that is again a tree. We let $N(T)$ denote the number of subtrees of a tree $T$. Moreover, let $\mu(T)$ be the average subtree order of $T$, i.e., the average number of vertices in a uniformly random subtree of $T$. Both of these quantities have been widely studied. In particular, there are many results on the maximum and minimum values of the number of subtrees \cite{SzW05,SW15,AWW13,AWW17,KW08,ZZ15} and the average subtree order \cite{jamison1983average,CWW21,MO19,H14,VW10} in different classes of trees. For example, it is known that the star has the greatest number of subtrees among all trees of order $n$ (namely $2^{n-1} + n-1$), while the path has the least number of subtrees (namely $\binom{n+1}{2}$). The average subtree order is also smallest for paths, with a value of $\frac{n+2}{3}$ (shown by Jamison \cite{jamison1983average}), while the trees that attain the maximum are not known for general $n$. However, it is known that the average subtree order can be close to the trivial upper bound $n$, and there is an intriguing conjecture, also due to Jamison, that the trees with the greatest average subtree order, given the number of vertices, are always caterpillars. See \cite{CWW21} for more on this topic.

Moreover, limit theorems have been established for both quantities under different random tree models \cite{CJ18,W15,RW19}. Our work presents a new viewpoint by studying asymptotic behaviour in a deterministic setting.

For both the number of subtrees and the average subtree order, we will prove convergence of suitably normalised versions. Our first main result is a natural analogue of the aforementioned results of Lyons on the number of spanning trees, and of Abért, Csikvári, Frenkel, and Kun on the number of matchings.
Specifically, we have the following theorem.

\begin{theorem} \label{thm:BSconstant_general}
    Let $T_n$ be a sequence of finite trees such that $|T_n| \to \infty$. If $T_n$ is Benjamini--Schramm convergent,
    then $\frac{\log N(T_n)}{\abs{T_n}}$ converges to a constant depending solely on the Benjamini--Schramm limit.
\end{theorem}

Our second main result concerns the subtree density, which is a normalised version of the average subtree order: the density $D(T)$ is defined as $D(T) = \mu(T)/|T|$. We will see that $D(T_n)$ also converges for a sequence $T_n$ of trees that converges in the Benjamini--Schramm sense, but only under a technical assumption that rules out long paths. This condition is essential, as there are otherwise counterexamples.

\begin{theorem} \label{main local conv theorem}
    Let $T_n$ be a sequence of finite trees such that $|T_n|\to \infty$ as $n\to \infty$. Suppose that $T_n$ is Benjamini--Schramm convergent, and that the limit measure $\sigma(\cdot,r)$ satisfies
\begin{equation}\label{eq:no_long_paths}
\lim_{r \to \infty} \sigma(P_{2r+1}^*,r) = 0,
\end{equation}
where $P_{2r+1}^*$ is the path with $2r+1$ vertices centred at the middle vertex.
Then $D(T_n)$ converges to a constant depending solely on the Benjamini--Schramm limit.
\end{theorem}

The technical condition~\eqref{eq:no_long_paths} forbids long paths: in the limit measure, the probability that the $r$-ball centred at a random vertex is the path $P_{2r+1}^*$ tends to $0$ as $r \to \infty$. As we will see (by exhibiting explicit counterexamples), this technical condition is not just an artifact of our proof, but is actually required. \cref{main local conv theorem} is proved in Section~\ref{sec:density} along the same lines as \cref{thm:BSconstant_general}, but we will need a few more auxiliary lemmas to establish it. 

Let us briefly describe the strategy to prove these theorems. 
Let $\mathcal{RG}$ be the set of  undirected, connected and locally finite graphs with a distinguished vertex called the root. Similarly, let $\mathcal{RT} \subset \mathcal{RG}$ be the set of undirected locally finite rooted trees. The \emph{rooted distance} (see e.g.~\cite{BS01,ACFK16}) of two rooted graphs is defined to be $2^{-k},$ where $k$ is the maximal integer such that the two $k$-balls centred at the two roots are isomorphic as rooted graphs (if there is no such $k$, the distance is $0$). The rooted distance turns $\mathcal{RG}$ into a complete metric space. 

The Benjamini--Schramm convergence of a sequence of finite graphs $G_n$ can be defined as the weak convergence of corresponding measures $\sigma_n$ on $\mathcal{RG}$ obtained by choosing a (uniformly) random root (see~\cite{ATV13}, Section 2.2), i.e.,
\begin{displaymath}
    G_n \xrightarrow[]{\text{local}} G \iff \sigma_n \xrightarrow[]{\text{weak}} \sigma.
\end{displaymath}

To obtain \cref{thm:BSconstant_general}, we  proceed as follows. We first establish a well-defined function $f:\mathcal{RT} \to \mathbb{R}$ such that 
$$\frac{1}{\abs{T_n}} \sum_{w \in T_n} f(T_n,w)$$
has the same limit (if it exists) as $\frac{\log N(T_n)}{\abs{T_n}}$ if $\abs{T_n} \to \infty$. This expression can be interpreted as the expected value $\mathbb{E}_{\sigma_n}(f)$ of $f$ with respect to the measure $\sigma_{n}$ on $\mathcal{RT}$ obtained by rooting $T_n$ at a random root. By the definition of weak convergence of measures, for continuous and bounded~$f$, we have $\mathbb{E}_{\sigma_n}(f) \to \mathbb{E}_{\sigma} (f)$. Since the limit $\mathbb{E}_{\sigma} (f)$ only depends on $\sigma,$ the Benjamini--Schramm limit of $T_n,$ \cref{thm:BSconstant_general} is proved by showing that $f$ is continuous and bounded. This is done in Section~\ref{sec:number}.
The proof of Theorem~\ref{main local conv theorem} follows the same lines, with a few added technicalities, and is presented in Section~\ref{sec:density}. The paper concludes in Section~\ref{sec:values} with a discussion of the possible values that the quantities in our two main theorems can take.

\section{Preliminaries}

We start with some preliminary considerations that will be needed for the proofs of both of our main theorems. Let $T$ be a finite tree.
We write $\mathcal{S} (T)$ for the set of subtrees of $T$ and $\mathcal{S} (T,v)$ for the set of subtrees of $T$ that contain $v$. Then $N(T)=\abs{\mathcal{S}(T)}$ is the number of subtrees of $T$, and we also write $N(T,v)=\abs{\mathcal{S}(T,v)}$ for the number of subtrees that contain $v$. This quantity can be computed in a recursive fashion, which will be important for us. Let $v_1, \dots,v_d$ be the neighbours of $v$, and let $T_1, \dots,T_d$ be the corresponding \emph{branches} (the components of $T-v$, rooted at the neighbours of $v$; we will also refer to these as the \emph{branches of} $v$ in the following). 

There is a trivial bijection between subtrees in $\mathcal{S}(T,v)$ and $d$-tuples of (possibly empty) subtrees in the Cartesian product  $\bigotimes_{i=1}^d  (\{\emptyset\} \cup \mathcal{S}(T_i, v_i))$. Thus we have
\begin{equation}\label{subtree_recursion}
    N(T,v) = \prod_{i = 1}^{d} (1+N(T_i,v_i)).
\end{equation}
A \emph{subtree core} of $T$, denoted by $c$ throughout this paper, is a vertex $v$ that maximizes $N(T, v)$. A finite tree can have one or two subtree cores, as shown by Sz\'ekely and Wang~\cite{BS01}. If there are two cores, they must be adjacent. The subtree core of a tree will often serve as a reference point for us in our proofs.

Note that the notion of a subtree core is no longer well-defined for a tree of infinite order. To sidestep this issue, we use a characterisation of the subtree core that is given in the following lemma. Let $\mathcal{S}_v (T,w)$ be the set of subtrees $S$ of $T$ that contain $w$ and ``grow away'' from vertex $v$, i.e., for all $u \in S$, the unique path $uv$ contains $w$. We finally set $N_v (T,w) = \abs{\mathcal{S}_v (T,w)}$.

Moreover, let us define $M(T,v):= \min_{e \ni v} N(T-e,v),$ where the minimum is taken over all edges $e$ incident with $v$. In words, this is the minimum number of subtrees of $T$ containing $v$ after removing one of the branches at $v$. Observe that this is meaningful for arbitrary locally finite trees (though the minimum might be infinite). The following lemma shows that $M(T,v)$ and $N_c(T,v)$ coincide for any finite tree $T$ unless $v = c$. In other words, the edge $e$ for which the minimum in the definition of $M(T,v)$ is attained is precisely the first edge on the path from $v$ to the subtree core $c$.
 
\begin{lemma}\label{lemma:N_and_M}
    Let $T$ be a finite tree and $c$ be a vertex of the subtree core. Let $v \neq c$ be a vertex of $T$. For a neighbour $u$ of $v$, let $T_u$ be the component of $u$ in $T - uv$. Suppose that $w$ is a neighbour of $v$ such that
    $$N(T_w,w) = \max_{u \in N(v)} N(T_u,u).   
    $$
Then $w$ lies in the component of $T - v$ that contains $c$. In particular, $N_c(T,v)=M(T,v).$
\end{lemma}

\begin{proof}
    Let $c,v$ and $w$ be as in the lemma. The case that $v$ is a leaf is trivial, so it is safe to assume that $v$ has at least two neighbours. Let $v_1,\dots,v_d$ be all the neighbours of $v$, and let $T_1,\dots,T_d$ be the corresponding branches of $v$. Without loss of generality, assume $c \in T_1$, so that $w=v_1,$ as shown in \cref{fig:coremax2}. As proven in~\cite{SzW05}, $N(T,\cdot)$ is increasing along any path from a leaf to a subtree core, so we have $N(T,v_1) \ge N(T,v)$. 
    With $n_i=N(T_i,v_i)$ for $i=1, \dots ,d$, one has, by~\eqref{subtree_recursion},
    \begin{align*}
        N(T,v_1) &= n_1 \left(1+(1+n_2)\ldots(1+n_d) \right),\\
        N(T,v) &= (1+n_1)(1+n_2)\ldots(1+n_d).
    \end{align*}
    Thus $N(T,v_1) \ge N(T,v)$ yields
    \begin{equation*}
        \frac{1+n_1}{n_1} \le \frac{1+(1+n_2)\ldots(1+n_d)}{(1+n_2)\ldots(1+n_d)}
    \end{equation*}
    or equivalently
    \begin{equation*}
        n_1 \ge (1+n_2)\ldots(1+n_d).
    \end{equation*}
    The last inequality implies $n_1 > n_i,$ for $i=2, \dots ,d.$ It immediately follows that the branch containing the subtree core is exactly the branch that maximizes $N(T_i,v_i),$ which in our case is $T_1.$ 
    Noting that
    \begin{equation*}
        N(T - T_i,v) = \prod_{\substack{j=1 \\ j \neq i}}^d (1 + N(T_j,v_j))
    \end{equation*}
    by~\eqref{subtree_recursion}, one finds that $N(T - T_i,v)$ is minimal if and only if $N(T_i,v_i)$ is maximal. Thus in particular, $N_c(T,v)=N(T-T_1,v)=M(T,v).$
\end{proof}

    \begin{figure}[htbp]
        \begin{center}    
            \begin{tikzpicture}     
        	{          
                 \node at (-1.85,-0.25) {$c$};
                \node at (-2.23,0) {$\ldots$};
                \draw[fill] (-1.8,0) circle (0.06); 
                \draw[thick] (-1.8,0)--(-1.45,0);
                \node at (-1.05,0) {$\ldots$};    
                \draw[fill] (0.05,0) circle (0.06);
        	    
                \node at (0.05,-0.28) {$w=v_1$};
                \draw[thick] (0.05,0)--(-0.55,0);
                \draw[thick] (-0.5,0)--(-0.6,0);

                \node[rotate=135] at (-0.15,0.3) {$\ldots$};
                
                \node at (1.5,0.7) {$v_2$};
                \node at (1.5,-0.7) {$v_d$};
                \draw[thick] (1.75,-0.45)--(1.1,0)--(1.75,0.45);

                \draw[thick] (1.1,0)--(0.05,0);
                \draw[fill] (1.1,0) circle (0.06);
                \node at (1.05,0.25) {$v$}; 
                \draw[fill] (1.75,-0.45) circle (0.06);   
                \node[scale=0.8] at (2.2,-0.48) {$\ldots$};
        	  \draw[fill] (1.75,0.45) circle (0.06);
                \node[scale=0.8] at (2.2,0.48) {$\ldots$};
                \node[rotate=90,scale=0.8] at (1.7,0) {$\ldots$};
                \draw[dashed] (2.5,0.1)--(1.75,0.45)--(2.5,0.9);
                \draw[dashed] (2.5,0.1)--(2.5,0.9);
                \node at (2.9,0.48) {$T_2$};
                \draw[dashed] (2.5,-0.1)--(1.75,-0.45)--(2.5,-0.9);
                \draw[dashed] (2.5,-0.1)--(2.5,-0.9);
                \node at (2.9,-0.48) {$T_d$};

                \draw[dashed] (-2.3,0.8)--(0.7,0.8);
                \draw[dashed] (-2.3,-0.8)--(0.7,-0.8);
                \draw[dashed] (0.7,0.8)--(0.7,-0.8);
                \node at (-2.6,0.8) {$T_1$};
                
        	     } 	
            \end{tikzpicture}\\ 
        \end{center} 
        \caption{The subtree core $c$ lies in the branch $T_1$ of $v_1.$}
        \label{fig:coremax2} 
    \end{figure}
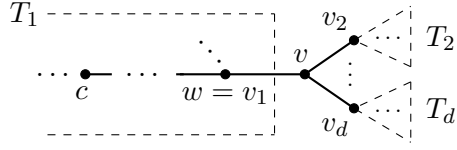

\section{Number of subtrees and Benjamini--Schramm convergence}
\label{sec:number}

\subsection{Proof of Theorem~\ref{thm:BSconstant_general}}

We start the proof of our first main theorem with two more auxiliary lemmas.

\begin{lemma}\label{lemma:core_estimate}
    Let $T$ be a tree of order $n$, and let $c$ be a subtree core of $T$. Then
    \begin{displaymath}
    \left| \log (1+N(T,c)) - \log N(T) \right| \leq \log 2.
    \end{displaymath}
\end{lemma}

\begin{proof}
    Theorem 6 in \cite{SzW14} (see also \cite[Theorem 5]{SzW13}) states, rewritten in our notation, that
$$\frac{N(T,c)}{N(T)} \geq \frac{\lfloor \frac{n}2 \rfloor + 1 }{2 \lfloor \frac{n}2 \rfloor + 1}.$$
Thus
$$\frac{1 + N(T,c)}{N(T)} > \frac{\lfloor \frac{n}2 \rfloor + 1 }{2 \lfloor \frac{n}2 \rfloor + 1} > \frac12.$$
On the other hand, it is trivial that
$$\frac{1 + N(T,c)}{N(T)} \leq \frac{2N(T,c)}{N(T)} \leq 2.$$
The inequality follows upon taking the logarithm.
\end{proof}

\begin{lemma}\label{lemma:subformula}
    Let $T$ be a finite tree and $v \in T.$ Then
    \begin{displaymath}
    \log \left(1+N(T,v) \right) = \sum_{w \in T} \log \left( 1 + \frac{1}{N_v(T,w)} \right).
    \end{displaymath}
\end{lemma}

\begin{proof}
    We use induction on the number of vertices of $T.$ When $T$ is the singleton tree, both sides equal $\log 2$. For the induction step, we apply~\eqref{subtree_recursion}.
    We have
    \begin{align*}
        \log \left(1+N(T,v) \right) &= \log N(T,v) + \log \frac{1+N(T,v)}{N(T,v)}\\
        &= \log \left(\prod_{i = 1}^{d} (1+N(T_i,v_i)) \right) + \log \left(1+ \frac{1}{N(T,v)} \right) \\
        &= \sum_{i=1}^{d} \bigg( \log (1+N(T_i,v_i)) \bigg) + \log \left(1+ \frac{1}{N(T,v)} \right).
    \end{align*}    
        Now, by the induction hypothesis,
    \begin{align*}
        \log \left(1+N(T,v) \right) &= \sum_{i=1}^{d} \left( \sum_{w \in T_i} \log \left( 1 + \frac{1}{N_{v_i}(T_i,w)} \right) \right) + \log \left(1+ \frac{1}{N(T,v)} \right) \\
        &= \sum_{w \in T-v} \Biggl( \log \left( 1 + \frac{1}{N_{v}(T,w)} \right) \Biggl) + \log \left(1+ \frac{1}{N_v(T,v)} \right) \\
        &= \sum_{w \in T} \log \left( 1 + \frac{1}{N_v(T,w)} \right),
    \end{align*}
    which completes the induction.
\end{proof}

Now consider a sequence $T_n$ of finite trees such that $\abs{T_n}\to \infty$. We let $c_n$ denote the subtree core of $T_n$ (if there are two, we pick either).
\cref{lemma:core_estimate} and \cref{lemma:subformula} together imply that
\begin{equation*}
    \left\lvert \log N(T_n) - \sum_{w \in T_n} \log \left(1 + \frac{1}{N_{c_n}(T_n,w)} \right) \right\rvert \leq \log 2.
\end{equation*}
We have $N_{c_n}(T_n,w) = M(T_n,w)$ for all $w \neq c_n$ by \cref{lemma:N_and_M}. Moreover, it is clear that
\begin{equation*}
    \left\lvert \log \left(1 + \frac{1}{N_{c_n}(T_n,c_n)} \right) - \log \left(1 + \frac{1}{M(T_n,c_n)} \right) \right\rvert \leq \log 2,
\end{equation*}
since both $N_{c_n}(T_n,c_n)$ and $M(T_n,c_n)$ are integers. Thus we have
\begin{equation*}
    \left\lvert \log N(T_n) - \sum_{w \in T_n} \log \left(1 + \frac{1}{M(T_n,w)} \right) \right\rvert \leq 2\log 2.
\end{equation*}
As $n \to \infty$, it now follows that
\begin{equation*}
    \left\lvert \frac{\log N(T_n)}{\abs{T_n}} - \frac{1}{\abs{T_n}} \sum_{w \in T_n} \log \left(1 + \frac{1}{M(T_n,w)} \right) \right\rvert \leq \frac{2\log 2}{\abs{T_n}} \to 0.
\end{equation*}
So the two expressions $\frac{\log N(T_n)}{\abs{T_n}}$ and $\frac{1}{\abs{T_n}} \sum_{w \in T_n} \log \left(1 + \frac{1}{M(T_n,w)} \right)$ converge to the same limit (if it exists).
We have thus established a function
    \begin{align*}
        f : \mathcal{RT} &\to \mathbb{R}, \\
        (T,w) &\mapsto \log \left( 1 + \frac{1}{M(T,w)} \right),
    \end{align*}
with the property that 
\begin{equation}\label{eq:limit_rep}
\lim_{n \to \infty} \frac{1}{\abs{T_n}} \sum_{w \in T_n} f(T_n,w)
= \lim_{n \to \infty} \frac{\log N(T_n)}{\abs{T_n}}
\end{equation}
for a sequence $T_n$ of finite trees such that $\abs{T_n} \to \infty.$, provided that the limit exists. 

Now with respect to the measure $\sigma_n$ on $\mathcal{RT}$ associated with $T_n$, the expression on the left of~\eqref{eq:limit_rep} is exactly the expectation of our function $f$, i.e.,  $\mathbb{E}_{\sigma_n}(f)$. Weak convergence of the measures $\sigma_n$ to a limit $\sigma$ implies that for every bounded, continuous function $g$, one has $\mathbb{E}_{\sigma_n}(g) \to \mathbb{E}_{\sigma}(g)$. Thus, \cref{thm:BSconstant_general} is proved by showing that $f$ is bounded and continuous.

Since $f$ is a real function defined on the metric space $\mathcal{RT}$, continuity of $f$ is equivalent to the condition that for every convergent sequence $x_n \to x$ in $\mathcal{RT}$, we have $f(x_n) \to f(x)$. Therefore, the next lemma yields the continuity of $f$.

\begin{lemma} \label{continuity of f} We have
    $\log (1+\frac{1}{M(T_n,v_n)}) \to \log (1+\frac{1}{ M(T,v)})$ if $(T_n,v_n) \to (T,v)$ in $\mathcal{RT}.$
\end{lemma}
\begin{proof}
    For a positive integer $r$, define $M_r(T,v):= M(T|_{v,r},v),$ where $T|_{v,r}$ is the (closed) $r$-ball centred at $v$ in $T.$ Convergence of $(T_n,v_n)$ to $(T,v)$ in $\mathcal{RT}$ implies that for any positive integer $r,$ the $r$-ball of $T_n$ centred at $v_n$ is eventually isomorphic to with the $r$-ball of $T$ centred at $v$ for sufficiently large $n \geq n_r$.
  
    Denote the neighbours of $v$ by $w_1, \ldots,w_d$, and their corresponding branches in $T-v$ by $S_1, \ldots, S_d$. Without loss of generality, assume that $N(S_1,w_1) = \max_j N(S_j,w_j)$. This readily implies that $M(T,v)=N(T-vw_1,v)$ (as in the proof of \cref{lemma:N_and_M}). 
    First we consider the case that at most one of the branches $S_j$ has infinitely many vertices. If there is such a branch, then it must be $T_1$ by our choice. Now for sufficiently large $r$, the $r$-ball $T|_{v,r}$ completely contains all branches except possibly $T_1$, since they are finite. Moreover, if $r$ is chosen sufficiently large, removing $vw_1$ yields the minimum in the definition of $M_r(T,v)$ (again by our choice of $T_1$). Thus we must have $M_r(T,v)=M(T,v)=N(T-vw_1,v)$. For sufficiently large $n$, the $r$-ball of $T_n$ centred at $v_n$ coincides with the $r$-ball of $T$ centred at $v$, so we have
    $$M(T_n,v_n) = M_r(T_n,v_n) = M_r(T,v)=M(T,v).$$
    On the other hand, in the case that at least two of the branches are infinite, removing an edge $vv_j$ only eliminates one of them, so we have $M(T,v)=\infty$. For every fixed $r$, each infinite branch must contain a vertex whose distance from $v$ is $r$. Thus $T|_{v,r}$ contains two vertices in distinct branches whose distance from $v$ is at least $r$. For large enough $n$, the $r$-ball $T_n|_{v_n,r}$ also has this property, which implies that $M(T_n,v_n) \geq r$. Therefore, we must have
        $\lim_{n \to \infty} M(T_n,v_n)=M(T,v)=\infty$, completing the proof also in this case.
\end{proof}

We can now conclude the proof of our first main result.

\begin{proof}[Proof of \cref{thm:BSconstant_general}]
\cref{continuity of f} shows that $f: \mathcal{RT} \to \mathbb{R}$ given by
$$f(T,w) = \log \left( 1 + \frac{1}{M(T,w)} \right)$$
 is a continuous function. Clearly, $f(T,w) \in [0,\log 2]$ for any $(T,w) \in \mathcal{RT}$, so $f$ is also bounded. Since the assumption of Benjamini--Schramm convergence implies weak convergence of the measures $\sigma_n$ associated with $T_n$ to a limit $\sigma$,
$$\mathbb{E}_{\sigma_n}(f) = \frac{1}{\abs{T_n}} \sum_{w \in T_n} f(T_n,w) \to \mathbb{E}_{\sigma}(f)$$
as $n \to \infty$, and the limit only depends on the Benjamini--Schramm limit of the sequence $T_n$. In view of~\eqref{eq:limit_rep}, this completes the proof.
\end{proof}

\subsection{Examples}\label{sec:examples1}

Let us illustrate \cref{thm:BSconstant_general} with three examples.

\paragraph{Paths.}
The simplest example is the sequence of paths $P_n$, where $|P_n|=n$. $P_n$ converges to the (two-sided) infinite path $P_\infty$ in the Benjamini--Schramm sense. More precisely, $P_\infty$ represents the limit measure $\sigma$ with mass concentrated at the infinite path rooted at any vertex $v$. For $P_n,$ $N(P_n)=\frac{n(n+1)}{2}$ and therefore $\lim_{n \to \infty} (\log N(P_n))/n = 0$. On the other hand, $\mathbb{E}_{\sigma}(f) = 0$ since $M(P_{\infty},v) = \infty$ and thus $f(P_{\infty},v) = 0$. \hfill \(\diamondsuit\)

\paragraph{Combs.}
Let $C_n$ denote the tree obtained by attaching a leaf to each vertex of $P_n$. Figure \ref{fig:combs} shows the first three combs $C_1,C_2,C_3$. Note that $|C_n|=2n$. It is not too difficult to prove (see the later discussion in Section~\ref{sec:examples2}) that $N(C_n) = 2^{n+2}-n-4$, thus $\lim_{n \to \infty} \frac{\log N(C_n)}{|C_n|} = \frac{\log 2}{2}$.

On the other hand, let $C_\infty$ be the (two-sided) infinite comb. The Benjamini--Schramm limit of $C_n$ is $C_{\infty}$, rooted either at a leaf or an internal vertex, each with probability $\frac12$. If $v$ is a leaf of $C_{\infty}$, then clearly $M(C_{\infty},v) = 1$ and $f(C_{\infty},v) = \log 2$. On the other hand, if $v$ is an internal vertex of $C_{\infty}$, then $M(C_{\infty},v) = \infty$ and $f(C_{\infty},v) = 0$. Thus $\mathbb{E}_{\sigma}(f) = \frac12 \log 2$ in this case, in agreement with \cref{thm:BSconstant_general}. \hfill \(\diamondsuit\)

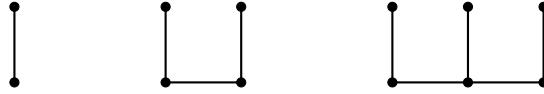
\begin{figure}[htbp]
        \begin{center}    
            \begin{tikzpicture}     
        	{                
                \draw[fill] (-6,0) circle (0.06);
                \draw[fill] (-6,1) circle (0.06);
                \draw[thick] (-6,0)--(-6,1);
                
                \draw[fill] (-4,0) circle (0.06);
                \draw[fill] (-3,0) circle (0.06);
                \draw[fill] (-4,1) circle (0.06);
                \draw[fill] (-3,1) circle (0.06);
                \draw[thick] (-4,0)--(-4,1);
                \draw[thick] (-3,0)--(-3,1);
                \draw[thick] (-4,0)--(-3,0);

                \draw[fill] (-1,0) circle (0.06);
                \draw[fill] (0,0) circle (0.06);
                \draw[fill] (1,0) circle (0.06);
                \draw[thick] (-1,0)--(0,0)--(1,0);
                \draw[fill] (-1,1) circle (0.06);
                \draw[fill] (0,1) circle (0.06);
                \draw[fill] (1,1) circle (0.06);
                \draw[thick] (-1,0)--(-1,1);
                \draw[thick] (0,0)--(0,1);
                \draw[thick] (1,0)--(1,1);
           
        	     } 	
            \end{tikzpicture}\\ 
        \end{center} 
        \caption{$C_1,C_2$ and $C_3.$ }
        \label{fig:combs} 
    \end{figure}

\paragraph{Regular trees (Bethe trees).}

As a more complicated example, let us consider the sequence of 3-regular trees $RT_n^3$, shown in Figure~\ref{fig:3regulartree}, and its Benjamini--Schramm limit $RT^3$. Such trees are also known as \emph{Bethe trees}. Subtrees of Bethe trees were studied in \cite{YLW15}. Let $L_n^k$ be the set of vertices in $RT_n^3$ whose distance from the closest leaf is $k$. Since $\abs{L^k_n}=3 \cdot 2^{n-k-1}$ ($0 \leq k \leq n-1$) and $\abs {RT_n^3}= 3 \cdot 2^n - 2$, we have $\lim_{n \to \infty} \frac{\abs{L^k_n}}{\abs{RT_n^3}} = 2^{-k-1}$. Thus the Benjamini--Schramm limit of $RT_n^3$ is the infinite canopy tree $RT^3$, shown in Figure \ref{fig:canopytree} (rather than the infinite $3$-regular tree, as one might intuitively guess), rooted at $x_k$ with probability $2^{-k-1}$.

    \begin{figure}[htbp]
        \begin{center}    
            \begin{tikzpicture}     
        	{          
                
                \draw[fill] (-10,0) circle (0.06);
                \node at (-10,-1.8) {$RT^3_0$};
                
                \draw[fill] (-7.5,0) circle (0.06);
                \node at (-7.5,-1.8) {$RT^3_1$};
                \draw[fill] (-6.9,-0.4) circle (0.06);
                \draw[fill] (-8.1,-0.4) circle (0.06);
                \draw[fill] (-7.5,0.65) circle (0.06);
                \draw[thick] (-6.9,-0.4)--(-7.5,0)--(-8.1,-0.4);
                \draw[thick] (-7.5,0.65)--(-7.5,0);
                \draw[dashed] (-7.5,0.65)--(-8.1,-0.4)--(-6.9,-0.4);
                \draw[dashed] (-7.5,0.65)--(-6.9,-0.4);
                
                \draw[fill] (-4,0) circle (0.06);
                \node at (-4,-1.8) {$RT^3_2$};
                \draw[fill] (-3.4,-0.4) circle (0.06);
                \draw[fill] (-4.6,-0.4) circle (0.06);
                \draw[fill] (-4,0.65) circle (0.06);
                \draw[thick] (-3.4,-0.4)--(-4,0)--(-4.6,-0.4);
                \draw[thick] (-4,0.65)--(-4,0);
                \draw[fill] (-4.4,1) circle (0.06);
                \draw[fill] (-3.6,1) circle (0.06);
                \draw[thick] (-4.4,1)--(-4,0.65)--(-3.6,1);
                \draw[fill] (-5,-0.1) circle (0.06);
                \draw[fill] (-4.8,-0.8) circle (0.06);
                \draw[thick] (-4.8,-0.8)--(-4.6,-0.4)--(-5,-0.1);
                \draw[fill] (-3,-0.1) circle (0.06);
                \draw[fill] (-3.2,-0.8) circle (0.06);
                \draw[thick] (-3.2,-0.8)--(-3.4,-0.4)--(-3,-0.1);
                \draw[dashed] (-3,-0.1)--(-3.2,-0.8)--(-4.8,-0.8);
                \draw[dashed] (-4.8,-0.8)--(-5,-0.1)--(-4.4,1);
                \draw[dashed] (-4.4,1)--(-3.6,1)--(-3,-0.1);
                
                \draw[fill] (0,0) circle (0.06);
                \node at (0,-1.8) {$RT^3_3$};
                \draw[fill] (0.6,-0.4) circle (0.06);
                \draw[fill] (-0.6,-0.4) circle (0.06);
                \draw[fill] (0,0.65) circle (0.06);
                \draw[thick] (0.6,-0.4)--(0,0)--(-0.6,-0.4);
                \draw[thick] (0,0.65)--(0,0);
                \draw[fill] (-0.4,1) circle (0.06);
                \draw[fill] (0.4,1) circle (0.06);
                \draw[thick] (-0.4,1)--(0,0.65)--(0.4,1);
                \draw[fill] (-1,-0.1) circle (0.06);
                \draw[fill] (-0.8,-0.8) circle (0.06);
                \draw[thick] (-0.8,-0.8)--(-0.6,-0.4)--(-1,-0.1);
                \draw[fill] (1,-0.1) circle (0.06);
                \draw[fill] (0.8,-0.8) circle (0.06);
                \draw[thick] (0.8,-0.8)--(0.6,-0.4)--(1,-0.1);
                \draw[fill] (0.8,-1.1) circle (0.06);
                \draw[fill] (1.1,-0.8) circle (0.06);
                \draw[thick] (0.8,-1.1)--(0.8,-0.8)--(1.1,-0.8);
                \draw[fill] (1.1,0.2) circle (0.06);
                \draw[fill] (1.3,-0.2) circle (0.06);
                \draw[thick] (1.1,0.2)--(1,-0.1)--(1.3,-0.2);
                \draw[fill] (-0.4,1.3) circle (0.06);
                \draw[fill] (-0.7,1) circle (0.06);
                \draw[thick] (-0.4,1.3)--(-0.4,1)--(-0.7,1);
                \draw[fill] (0.4,1.3) circle (0.06);
                \draw[fill] (0.7,1) circle (0.06);
                \draw[thick] (0.4,1.3)--(0.4,1)--(0.7,1);
                \draw[fill] (-0.8,-1.1) circle (0.06);
                \draw[fill] (-1.1,-0.8) circle (0.06);
                \draw[thick] (-0.8,-1.1)--(-0.8,-0.8)--(-1.1,-0.8);
                \draw[fill] (-1.1,0.2) circle (0.06);
                \draw[fill] (-1.3,-0.2) circle (0.06);
                \draw[thick] (-1.1,0.2)--(-1,-0.1)--(-1.3,-0.2);
                \draw[dashed] (-1.1,0.2)--(-1.3,-0.2)--(-1.1,-0.8);
                \draw[dashed] (-1.1,-0.8)--(-0.8,-1.1)--(0.8,-1.1);
                \draw[dashed] (0.8,-1.1)--(1.1,-0.8)--(1.3,-0.2);
                \draw[dashed] (1.3,-0.2)--(1.1,0.2)--(0.7,1);
                \draw[dashed] (0.7,1)--(0.4,1.3)--(-0.4,1.3);
                \draw[dashed] (-0.4,1.3)--(-0.7,1)--(-1.1,0.2);

        	     } 	
            \end{tikzpicture}\\ 
        \end{center} 
        \caption{First four trees in the sequence $RT^3_n$, with dashed lines indicating the sets $L^0_n$.}
        \label{fig:3regulartree} 
    \end{figure}
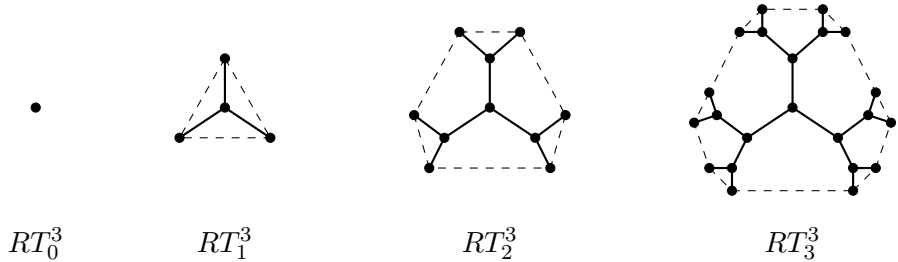

Now set $m_k = M(RT^3,x_k)$. If we remove the edge between $x_k$ and $x_{k+1}$ (the edge to the infinite component), the component of $x_k$ is a complete binary tree $B_k$ of height $k$, rooted at $x_k$: $B_0$ is a single vertex, and $B_k$ is obtained from two copies of $B_{k-1}$ by attaching their roots to a new common root. We thus have $m_0 = 1$ and, by~\eqref{subtree_recursion}, $m_{k+1} = (1+m_k)^2$ for $k \geq 0$. The resulting sequence $1,4,25,676,\ldots$ (A004019 in the OEIS \cite{OEIS}) is well-studied; Aho and Sloane \cite[Example 2.2]{AS} proved that $m_k$ is the nearest integer to $b^{2^k} - 1$ for a constant 
$$b = \exp \Big( \sum_{k \geq 0} 2^{-k} \log \Big( 1 + \frac{1}{m_k} \Big) \Big) \approx 2.2585184506.$$
So for the limit measure $\sigma$, we have
\begin{equation*}
    \mathbb{E}_{\sigma}(f) = \sum_{k \geq 0} 2^{-k-1} \log \big( 1 + \tfrac{1}{m_k} \big) = \frac{\log b}{2} \approx 0.4073545227.
\end{equation*}

    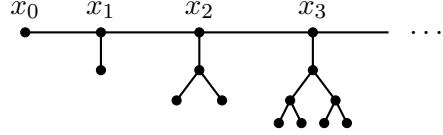
\begin{figure}[htbp]
        \begin{center}    
            \begin{tikzpicture}     
        	{
                \draw[thick] (-6.8,0.5)--(-5,0.5)--(-3,0.5)--(-2,0.5);
                \node at (-1.5,0.5) {$\ldots$};
                
                \draw[fill] (-6.8,0.5) circle (0.06);
                \node at (-6.8,0.8) {$x_0$};

                \node at (-5.8,0.8) {$x_1$};
                \draw[fill] (-5.8,0.5) circle (0.06);
                \draw[fill] (-5.8,0) circle (0.06);
                \draw[thick] (-5.8,0.5)--(-5.8,0);

                \node at (-4.5,0.8) {$x_2$};
                \draw[fill] (-4.5,0.5) circle (0.06);
                \draw[fill] (-4.5,0) circle (0.06);
                \draw[thick] (-4.5,0.5)--(-4.5,0);
                \draw[thick] (-4.8,-0.4)--(-4.5,0)--(-4.2,-0.4);
                \draw[fill] (-4.8,-0.4) circle (0.06);
                \draw[fill] (-4.2,-0.4) circle (0.06);

                \node at (-3,0.8) {$x_3$};
                \draw[fill] (-3,0.5) circle (0.06);
                \draw[fill] (-3,0) circle (0.06);
                \draw[thick] (-3,0)--(-3,0.5);
                \draw[thick] (-3.3,-0.4)--(-3,0)--(-2.7,-0.4);
                \draw[fill] (-3.3,-0.4) circle (0.06);
                \draw[fill] (-2.7,-0.4) circle (0.06);
                \draw[fill] (-3.45,-0.7) circle (0.06);
                \draw[fill] (-3.15,-0.7) circle (0.06);
                \draw[thick] (-3.45,-0.7)--(-3.3,-0.4)--(-3.15,-0.7);
                \draw[fill] (-2.85,-0.7) circle (0.06);
                \draw[fill] (-2.55,-0.7) circle (0.06);
                \draw[thick] (-2.85,-0.7)--(-2.7,-0.4)--(-2.55,-0.7);
             
        	     } 	
            \end{tikzpicture}\\ 
        \end{center} 
        \caption{$RT^3$ is the infinite canopy tree.}
        \label{fig:canopytree} 
    \end{figure}

On the other hand, one can obtain the explicit formula
\begin{equation*}
   N(RT_n^3) = (1+m_{n-1})^3 + \sum_{k=0}^{n-1} 3 \cdot 2^{n-k-1} m_k.
\end{equation*} 
The first term stands for the number of subtrees of $RT_n^3$ containing the centre (making use of~\eqref{subtree_recursion} again), while the $k$-th term in the sum is the number of subtrees where the innermost vertex belongs to $L^k_n$. In view of the rapid growth of the sequence $m_k$, the first term dominates the asymptotic behaviour (note also that the centre of $RT_n^3$ is the subtree core), and we have
\begin{equation*}
\lim_{n \to \infty} \frac{\log N(RT_n^3)}{|RT_n^3|} = \lim_{n \to \infty} \frac{\log( (1+m_{n-1})^3)}{3 \cdot 2^n-2} = \lim_{n \to \infty} \frac{\log (1+m_{n-1})}{2^n} = \frac{\log b}{2},
\end{equation*}
again confirming \cref{thm:BSconstant_general}. \hfill \(\diamondsuit\)

\section{Subtree density and Benjamini--Schramm convergence}\label{sec:density}

\subsection{Proof of Theorem~\ref{main local conv theorem}}

Our aim in this section is to prove convergence of $D(T_n)$ under the conditions of Theorem~\ref{main local conv theorem}. Before we proceed with the proof, let us explain why the condition on the limit measure $\sigma$ is required. Consider the sequence of paths $P_n$, and compare this to a sequence of trees $T_n$ that is constructed as follows: $T_n$ consists of a path with $n$ vertices, with a number $\lambda(n)$ of leaves attached to each of the two ends such that $\lambda(n) = \omega(\log n)$ but $\lambda(n) = o(n)$. These trees (known as \emph{double brooms}) occur frequently in the study of the mean subtree order and the subtree density.

Observe that both sequences have the same Benjamini--Schramm limit: the two-sided infinite path $P_{\infty}$ corresponding to the measure $\sigma$ with
$$\sigma(P_{2r+1}^*,r) = 1$$
for all $r$.

The subtree densities, however, have different limits: as mentioned in the introduction, the mean subtree order $\mu(P_n)$ of an $n$-vertex path is $\frac{n+2}{3}$, so the subtree density is $D(P_n) = \frac{n+2}{3n}$, which tends to $\frac13$. On the other hand, the number of subtrees of $T_n$ that contain both ends of the path (and thus the entire path between them) is $2^{2\lambda(n)}$, while the number of all other subtrees is $O(n 2^{\lambda(n)})$. So subtrees that contain the entire path dominate the rest, and all of them contain at least $n$ of the $n  + 2\lambda(n)$ vertices. It follows from these observations that $D(T_n) \to 1$ as $n \to \infty$. 

In fact, by modifying the construction, it is even possible to achieve any limit in $[\frac13,1]$ for the density while keeping the same Benjamini--Schramm limit. Fix a positive constant $x$, and let $T_n$ be the double broom consisting of a path of $p = \lfloor x 2^n \rfloor$ vertices and $n$ leaves attached at each end. This double broom has
\begin{itemize}
\item $2^{2n}$ subtrees that contain the entire path, and these subtrees have $p + n$ vertices on average;
\item $2(p-1) \cdot 2^{n}$ subtrees that contain exactly one end of the path. These subtrees have $\frac{p+n}{2}$ vertices on average.
\item $\binom{p-1}{2}$ subtrees consisting of part of the path, but containing neither end. These subtrees have $\frac{p}{3}$ vertices on average.
\item $2n$ subtrees consisting of a single leaf.
\end{itemize}
 It is not difficult to show from this that the density of $T_n$ converges to
$$\frac{x^2+6x+6}{3x^2+12x+6},$$
which attains all possible values in the interval $(\frac13,1)$ if we let $x$ vary in $(0,\infty)$. We will return to this discussion of the possible values in the following section.

These examples illustrate why long paths need to be ruled out. So throughout the rest of this section, we assume that $T_n$ is a sequence of finite trees that satisfies the assumptions of Theorem~\ref{main local conv theorem}: $\abs{T_n} \to \infty$ as $n \to \infty$, $T_n$ converges in the Benjamini--Schramm sense, and the limit measure $\sigma$ has the property that
$$\lim_{r \to \infty} \sigma(P_{2r+1}^*,r) = 0.$$
For finite $T$, the subtree density $D(T)$ of $T$ can be calculated as
$$D(T)=\frac{\mu(T)}{|T|}=\frac{1}{|T|} \sum_{v\in T} p(T,v),$$
where $p(T,v) = N(T,v)/N(T)$ is the probability that a random subtree contains $v$. So we would like to apply the same approach as in the previous section, with $p(T,v)$ playing the role of $f(T,v)$.

However, in order to extend $p(T,v)$ to infinite trees in a well-defined way, we have to replace it by a different quantity $q(T,v)$ first.

For a subset $A$ of vertices of a finite tree $T$, let $S_A$ be the event that a random subtree (picked uniformly at random among all subtrees of $T$) contains $A$. In particular, we let $S_v$ be the event that a random subtree contains the vertex $v$, and let $S_{v,w}$ be the event that a random subtree contains both the vertices $v$ and $w$. Note that with this notation, $p(T,v)=\mathbb{P}(S_v)$.

Now let $v$ be any vertex of $T$, and let $c$ be the subtree core of $T$ (if there are two, pick the one that is closer to $v$; we keep this rule throughout the section). Now we define
$$q(T,v):=\mathbb{P}(S_{v,c}|S_c)=\mathbb{P}(S_v|S_c).$$
In words, this is the probability that a random subtree, conditioned on containing the subtree core $c$, also contains $v$. Since most subtrees should contain the subtree core, we can expect $q(T,v)$ to be close to $p(T,v)$. This will be made precise later. Before we show that $p(T,v)$ can in fact be replaced by $q(T,v)$, we need some more auxiliary definitions.

Let $T$ be a finite tree, possibly with a distinguished root. A \emph{path segment} of $T$ is a path with the property that all its internal vertices have degree $2$ in $T$, while each of the ends is either the root or a vertex whose degree is not $2$. We say that a vertex \emph{lies in} a path segment if it is one of its internal vertices.

For a positive integer $m$, we call a finite (possibly rooted) tree $T$ an $m$-\emph{good tree} if at most half of the vertices of $T$ lie in path segments of length greater than $m$.

Trees that are not $m$-good trees have many long path segments and thus many vertices whose $r$-neighbourhood (i.e., the closed $r$-ball around the vertex) is $P_{2r+1}^*$. This is quantified in our next lemma.

\begin{lemma}\label{lem:many_path-nbhds}
Suppose that the (possibly rooted) tree $T$ is not an $m$-good tree. Then, for every positive integer $r$ with $2r-1 < m$, $T$ has at least $\frac12(1 - \frac{2(r-1)}{m}) |T|$ vertices for which the $r$-neighbourhood is isomorphic to $P_{2r+1}^*$.
\end{lemma}
\begin{proof}
A path segment of length $s > m$ has $s-1$ internal vertices. Except for the first and last $r-1$ of these vertices, each of them has an $r$-neighbourhood that is isomorphic to $P_{2r+1}^*$. The proportion of these vertices among all internal vertices of the path segment is thus at least
$$\frac{s-1-2(r-1)}{s-1} = 1 - \frac{2(r-1)}{s-1} \geq 1 - \frac{2(r-1)}{m}.$$
Since there are more than $\frac12|T_n|$ vertices in segments of length $> m$, we conclude that there are at least
$\frac12(1 - \frac{2(r-1)}{m}) |T|$ vertices whose $r$-neighbourhood is isomorphic to $P_{2r+1}^*$. 
\end{proof}

Now we show that every sequence of trees that satisfies the conditions of Theorem~\ref{main local conv theorem} eventually consists of $m$-good trees for suitable $m$.

\begin{lemma}\label{lemma-m-delta}
Suppose that $T_n$ is a sequence of finite trees that satisfies the conditions of Theorem~\ref{main local conv theorem}. Then, there exist positive integers $m$ and $N$ such that $T_n$ is an $m$-good tree for all $n\ge N$.
\end{lemma}

\begin{proof}
Suppose that there are no such $m$ and $N$. Thus, for every positive integer $m$, there are infinitely many indices $n$ for which $T_n$ is not an $m$-good tree. Let us apply this with $m = 4(r-1)$. By Lemma~\ref{lem:many_path-nbhds}, $T_n$ must have at least $\frac{|T_n|}{4}$ vertices for which the $r$-neighbourhood is isomorphic to $P_{2r+1}^*$. Since this is true for infinitely many $n$, we must have $\sigma(P_{2r+1}^*,r) \geq \frac14$ in the limit measure. As $r$ was arbitrary, this contradicts our assumptions.
\end{proof}

The following lemma shows that $m$-good trees need to have linearly many leaves.

\begin{lemma}\label{lemma: existence of C}
For every positive integer $m$, there exists a constant $C=C(m)$ such that every (possibly rooted) $m$-good tree $T$ has at least $C|T|$ leaves (excluding the root even if it is a leaf).
\end{lemma}

\begin{proof}
Let $T$ be an $m$-good tree, and let $k$ be the number of leaves of $T$. We perform two contractions of the path segments in $T$ that do not alter the number of leaves. First, we contract all path segments longer than $m$ to length~$1$. The resulting tree contains at least $\frac{|T|}{2}$ vertices by the definition of an $m$-good tree. Next, we contract all path segments (regardless of length) to length~$1$ and let $T^*$ denote the final tree. Since the tree before the second contraction has at least $\frac{|T|}{2}-1$ edges and only contains path segments whose length is at most $m$, $T^*$ has at least $(\frac{|T|}{2}-1)/m$ edges, thus at least $1+(\frac{|T|}{2}-1)/m \geq \frac{|T|}{2m}$ vertices. Note that $T^*$ does not have vertices of degree~$2$, except possibly for the root if there is one. Applying the degree sum formula to $T^*$, we obtain
$$2(|T^*|-1) = \sum_{v\in T^*} \deg(v) \ge k + 1 + 3(|T^*|-k-1),$$
which implies that
$$k\ge \frac{|T^*|}{2} \ge \frac{|T|}{4m}.$$ 
So the statement of the lemma holds with $C=\frac{1}{4m}$.
\end{proof}

Combining Lemma~\ref{lemma: existence of C} and Lemma~\ref{lemma-m-delta}, we obtain the following corollary.

\begin{corollary}\label{cor:many_leaves}
    Suppose that $T_n$ is a sequence of finite trees that satisfies the conditions of Theorem~\ref{main local conv theorem}. Then, there exists a constant $C > 0$ such that $T_n$ has at least $C|T_n|$ leaves for sufficiently large $n$.
\end{corollary}

Next, we prove the following lemma, which states that the averages of $p(T_n,v)$ and $q(T_n,v)$ converge to the same limit under our assumptions.

\begin{lemma} \label{lemma: p and q close enough}
    Suppose that $T_n$ is a sequence of finite trees that satisfies the conditions of Theorem~\ref{main local conv theorem}. Then we have
    $$\frac{1}{|T_n|}  \Big|\sum_{v\in T_n} p(T_n,v)-\sum_{v\in T_n} q(T_n,v) \Big| \to 0$$
    as $n \to \infty.$
\end{lemma}
\begin{proof}
The following estimate holds for an arbitrary tree $T$
with an arbitrary vertex $v$ and subtree core $c$:
    \begin{align*}
    \big| p(T,v) - q(T,v) \big| &= \big| \mathbb{P}(S_v)- \mathbb{P}(S_v|S_c) \big| \\
    &= \big| \mathbb{P}(S_v|S_c)\mathbb{P}(S_c)+\mathbb{P}(S_v|\overline{S_c})\mathbb{P}(\overline{S_c})- \mathbb{P}(S_v|S_c) \big| \\
    &= \big|\mathbb{P}(S_v|\overline{S_c})\mathbb{P}(\overline{S_c})-\mathbb{P}(S_v|S_c)\mathbb{P}(\overline{S_c}) \big| \\
    &= \big|\mathbb{P}(S_v|\overline{S_c})-\mathbb{P}(S_v|S_c) \big| \mathbb{P}(\overline{S_c}) \le \mathbb{P}(\overline{S_c}).
    \end{align*}
In words, the absolute difference between $p(T,v)$ and $q(T,v)$ is bounded above by the probability that a random subtree does not contain the core $c$.

By \cite[Lemma 5.1]{RW18}, we have
$$\mathbb{P}(\overline{S_c}) \leq |T|2^{-\frac{L(T)}{2}},$$
where $L(T)$ is the number of leaves of $T$. We can now apply this to $T_n$ and its subtree core $c_n$, giving us
$$\frac{1}{|T_n|}  \Big|\sum_{v\in T_n} p(T_n,v)-\sum_{v\in T_n} q(T_n,v) \Big| \leq \frac{1}{|T_n|}  \sum_{v\in T_n} \big| p(T_n,v)- q(T_n,v) \big| \leq 
\mathbb{P}(\overline{S_{c_n}}) \leq |T_n|2^{-\frac{L(T_n)}{2}}.$$
However, by Corollary~\ref{cor:many_leaves}, $L(T_n) \geq C|T_n|$ for sufficiently large $n$, so
$$\frac{1}{|T_n|}  \Big|\sum_{v\in T_n} p(T_n,v)-\sum_{v\in T_n} q(T_n,v) \Big| \leq |T_n|2^{-C|T_n|/2}.$$
Since $|T_n|$ tends to $\infty$, the statement follows.
\end{proof}

Our next aim is to extend the definition of $q(T,v)$ to arbitrary (including infinite) trees. First, for a finite tree $T$ and a vertex $v$, let $c$ be the subtree core of $T$ (the one that is closer to $v$ if there are two). Let $v = v_0,v_1,\ldots,v_k = c$ be the unique path from $v$ to $c$. By the chain rule for conditional probabilities, we have
$$q(T,v)= \mathbb{P}(S_{v,c}|S_c) = \mathbb{P}(S_{v_0,c}|S_{v_{k},c}) = \prod_{i = 0}^{k-1} \mathbb{P}(S_{v_i,c}|S_{v_{i+1},c}).$$
The factor $\mathbb{P}(S_{v_i,c}|S_{v_{i+1},c})$ represents the conditional probability that a subtree contains $v_i$ and $c$ (thus also all vertices in between) given that it contains $v_{i+1}$ and $c$ (and all vertices in between). Let $T_{v_i}$ be the component of $T - v_iv_{i+1}$ that contains $v_i$. A subtree of $T - T_{v_i}$ that contains $c$ and $v_{i+1}$ can be extended to a subtree that also contains $v_i$ by adding any subtree of $T_{v_i}$ that contains $v_i$. There are $N(T_{v_i},v_i)$ possibilities to do so, so we have
$$\mathbb{P}(S_{v_i,c}|S_{v_{i+1},c}) = \frac{N(T_{v_i},v_i)}{N(T_{v_i},v_i) + 1}.$$
We know from Lemma~\ref{lemma:N_and_M} that $N(T_{v_i},v_i) = M(T,v_i)$ for all $i < k$. Thus it follows that
$$q(T,v) = \prod_{i = 0}^{k-1} \mathbb{P}(S_{v_i,c}|S_{v_{i+1},c}) = \prod_{i = 0}^{k-1} \frac{M(T,v_i)}{M(T,v_i)+1}.$$
Let us set 
\begin{equation}\label{eq:qi_star}
    q_i^*(T,v) = \frac{M(T,v_i)}{M(T,v_i)+1}
\end{equation}
for $i < k$ and $q_i^*(T,v) = 1$ for $i \geq k$. Then we can conclude that
\begin{equation}\label{eq:q-product}
q(T,v) = \prod_{i \geq 0} q_i^*(T,v).
\end{equation}
Now we extend this to infinite trees in a similar fashion. For a vertex $v$ of an infinite tree $T$, we define a (finite or infinite) sequence $v_0,v_1,\ldots$ recursively as follows: first, $v_0 = v$. Now if $T - v_j$ has a unique infinite component, then let $v_{j+1}$ be $v_j$'s neighbour in that component (i.e., we always move towards the infinite component). Otherwise (i.e., if there are two or more infinite components), the sequence stops with $v_j$. We still define $q_i^*(T,v)$ by~\eqref{eq:qi_star} if the sequence does not stop with $v_i$. If $v_i$ is the last vertex in the sequence, then $q_j^*(T,v) = 1$ for all $j \geq i$. Since $M(T,v_i) = \infty$ if $T - v_i$ has more than one infinite component, this is also consistent with~\eqref{eq:qi_star}.

Now, we can extend $q(T,v)$ to infinite trees $T$ by means of~\eqref{eq:q-product}. It will be shown later that this product converges under suitable assumptions, but we will actually work with a truncated version. Before we can do so, we need one more technical lemma.
 
\begin{lemma}\label{lemma: product uni conv}
Suppose that $T_n$ is a sequence of finite trees that satisfies the conditions of Theorem~\ref{main local conv theorem}. Then, for every $\epsilon>0$, there exist a positive integer $N$ and a positive constant $C$ such that every tree $T_n$ with $n \geq N$ has a subset of vertices $U_n$ with the following two properties:
    \begin{enumerate}
        \item $\frac{|U_n|}{|T_n|} >1-\epsilon,$ and 
        \item $\left| q_i^*(T_n,u)-1 \right| \leq 2^{-Ci}$ for every $i \geq 0$ and every vertex $u\in U_n$.
    \end{enumerate}
\end{lemma}

\begin{proof}
The construction proceeds in several steps. First, since $\sigma(P_{2r+1}^*,r) \to 0$ as $r \to \infty$ by assumption, we know that $\sigma(P_{2r+1}^*,r) < \frac{\epsilon}{4}$ if we choose $r$ large enough. Since $\sigma(P_{2r+1}^*,r)$ is the limit of the probability that a random vertex of $T_n$ has an $r$-neighbourhood that is isomorphic to $P_{2r+1}^*$, we can find a positive integer $N$ such that $T_n$ contains less than $\frac{\epsilon |T_n|}{4}$ vertices for which the $r$-neighbourhood is isomorphic to $P_{2r+1}^*$.

Consider any tree $T_n$ with $n \geq N$. For a vertex $v$ of $T_n$, let $c$ be the subtree core (if there are two, the one closer to $v$), and let $T_n^{(v)}$ be the subtree of $T_n$ that consists of $v$ and all vertices for which the path to $c$ passes through $v$. If there is only one subtree core, $T_n^{(c)}$ is the whole tree $T_n$. If there are two subtree cores $c_1$ and $c_2$, then $T_n^{(c_1)}$ and $T_n^{(c_2)}$ are the two components of $T - c_1c_2$. Now, set $m = 4(r-1)$. We say that a vertex $v$ is \emph{bad} if $T_n^{(v)}$, rooted at $v$, is not an $m$-good tree. We also call the associated rooted tree $T_n^{(v)}$ a \emph{bad subtree} in this case. Now let us define $U_n$ to be the set of all vertices that \emph{do not lie in any bad subtree}.

The bad subtrees of $T_n$ form a poset with respect to inclusion. Note that any two such trees are either disjoint, or one contains the other. So we can find a set $S_1,S_2,\ldots,S_m$ of maximal elements in this poset such that every vertex that lies in a bad subtree is contained in exactly one $S_i$. In order to verify item (i), we need to show that
$$|T_n| - |U_n| = \sum_{i = 1}^m |S_i| < \epsilon |T_n|.$$
Here, we use Lemma~\ref{lem:many_path-nbhds}: each $S_i$ contains at least $\frac14 |S_i|$ vertices whose $r$-neighbourhood in $S_i$ is isomorphic to $P_{2r+1}^*$. This still applies in $T_n$: note that in the proof of Lemma~\ref{lem:many_path-nbhds}, the relevant $r$-neighbourhood never contains the root as an internal vertex. Hence $T_n$ has at least
$$\frac14 \sum_{i = 1}^m |S_i| = \frac14 (|T_n| - |U_n|)$$
such vertices. On the other hand, we know that this number is less than $\frac{\epsilon |T_n|}{4}$ by our choice of $r$ and $N$, implying that indeed $|T_n| - |U_n| < \epsilon |T_n|$.

It remains to prove (ii). Let $u$ be any vertex of $U_n$. The path to the subtree core is $u = v_0,v_1,\ldots, v_k = c$. The inequality in (ii) is trivial if $i \geq k$ and thus $q_i^*(T_n,u) = 1$, so we only consider $i < k$.

Since $u$ does not lie in a bad subtree by definition, none of the $v_i$ is a bad vertex. This means that $T_n^{(v_i)}$ is an $m$-good tree. By Lemma~\ref{lemma: existence of C}, $T_n^{(v_i)}$ has at least $C|T_n^{(v_i)}|$ leaves, excluding $v_i$ itself even if it is a leaf of $T_n^{(v_i)}$. It follows that the number of subtrees of $T_n^{(v_i)}$ containing $v_i$, which is exactly $M(T_n,v_i)$ by Lemma~\ref{lemma:N_and_M}, is at least $2^{C|T_n^{(v_i)}|}$ (consider all subtrees that contain all vertices of $T_n^{(v_i)}$ except possibly for some of these leaves). Since trivially $\abs{T_n^{(v_i)}} \geq i$, it follows that
$$M(T_n,v_i) \geq 2^{C|T_n^{(v_i)}|} \geq 2^{Ci},$$
and thus
$$0 \leq 1 - q_i^*(T_n,u) = 1 - \frac{M(T_n,v_i)}{M(T_n,v_i) + 1} = \frac{1}{1 + M(T_n,v_i)} \leq \frac{1}{2^{Ci}},$$
which completes the proof.
\end{proof}

\begin{lemma}\label{lemma: finite sum good enough}
Suppose that $T_n$ is a sequence of finite trees that satisfies the conditions of Theorem~\ref{main local conv theorem}. Then, for every $\epsilon>0$, there exist positive integers $N$ and $K$ such that 
$$\left| \sum_{v \in T_n} q(T_n,v)-\sum_{v \in T_n} \prod_{i=0}^K q_i^*(T_n,v) \right| <\epsilon |T_n|$$
holds for all $n>N$.
\end{lemma}

\begin{proof}
    We use Lemma~\ref{lemma: product uni conv} to determine $N$ and $K$ in the following way. First, choose $N$ according to Lemma~\ref{lemma: product uni conv} with $\epsilon$ replaced by $\frac{\epsilon}{2}$. Thus there exists a choice of a vertex set $U_n$ for every $n \geq N$ such that, for some positive constant $C$,
    \begin{enumerate}
        \item $|T_n\setminus U_n|\leq \frac{\epsilon}{2}|T_n|,$
        \item $|1-q^*_i(T_n,u)| \leq 2^{-Ci}$ for every $u\in U_n$ and every $i \geq 0$.
    \end{enumerate}
    Since the sum $\sum_{i \geq 1} 2^{-Ci}$ converges, so does $\prod_{i \geq 1} (1-2^{-Ci})$ by the standard convergence criterion for infinite products. Thus we can choose $K>0$ such that
$$1 > \prod_{i = K+1}^{\infty}  \left( 1- 2^{-Ci} \right) > 1-\frac{\epsilon}{2}.$$
    With $N$ and $K$ determined, consider any $n \geq N$. We have
    \begin{align*}
        \Big| \sum_{v \in T_n} q(T_n,v) &-\sum_{v \in T_n} \prod_{i=0}^K q_i^*(T_n,v) \Big| \\ 
        &= \Big| \sum_{v \in U_n} \Big( q(T_n,v)-\prod_{i=0}^K q_i^*(T_n,v) \Big) + \sum_{v \in T_n\setminus U_n} \Big( q(T_n,v)-\prod_{i=0}^K q_i^*(T_n,v) \Big)\Big| \\
        &\leq \Big| \sum_{v \in U_n} \Big( q(T_n,v)-\prod_{i=0}^K q_i^*(T_n,v) \Big) \Big| + \Big| \sum_{v \in T_n\setminus U_n} \Big( q(T_n,v)-\prod_{i=0}^K q_i^*(T_n,v) \Big)\Big|. \\
    \end{align*}
    
    For the first term, we recall that $q(T_n,v)=\prod_{i \geq0} q_i^*(T_n,v)$ and $0<q_i^*(T_n,v)\leq 1$ by definition. Thus,
    \begin{align*}
        \Big| \sum_{v \in U_n} \Big( q(T_n,v)-\prod_{i=0}^K q_i^*(T_n,v) \Big) \Big| &= \Big| \sum_{v \in U_n} \prod_{i=0}^{K} q_i^*(T_n,v) \Big( \prod_{i=K+1}^{\infty}q_i^*(T_n,v)- 1 \Big) \Big| \\
        &= \sum_{v \in U_n} \prod_{i=0}^{K} q_i^*(T_n,v) \Big( 1- \prod_{i=K+1}^{\infty}q_i^*(T_n,v) \Big) \\
        &\leq \sum_{v \in U_n}  \Big( 1- \prod_{i=K+1}^\infty \Big( 1-2^{-Ci} \Big) \Big) \\
        &= |U_n| \Big( 1-\prod_{i=K+1}^\infty \Big( 1-2^{-Ci} \Big) \Big)\\
        &\leq |T_n| \Big( 1- \Big(1-\frac{\epsilon}{2}\Big) \Big) \\
        &= \frac{\epsilon}{2}|T_n|.
    \end{align*}

    For the second term, observe that each summand has absolute value at most $1,$ as both $q(T_n,v)$ and $\prod_{i=0}^K q_i^*(T_n,v)$ take values in $(0,1].$ In addition, the number of summands is $|T_n\setminus U_n|\leq \frac{\epsilon}{2}|T_n|.$ Therefore, the second term is at most $\frac{\epsilon}{2}|T_n|.$

    Combining the two terms yields the desired inequality.
\end{proof}

As the last piece of the proof of \cref{main local conv theorem}, we show the continuity of $q_i^*$ for every fixed $i$, following the arguments in the proof of \cref{continuity of f}.

\begin{lemma}\label{lem:cont_qi}
    Suppose that $(T_n,v^{(n)}) \to (T,v)$ in $\mathcal{RT}.$ Then for any fixed index $i,$ $q_i^*(T_n,v^{(n)}) \to q_i^*(T,v).$
\end{lemma}

\begin{proof}
    If $T$ is of finite order, the proof is trivial as the sequence has to become constant and isomorphic to $(T,v)$ from some $n$ onward. Therefore, we assume $T$ to be a tree of infinite order.
    
\begin{enumerate}
    \item First, we consider the case that $q_i^*(T,v)=1.$ According to the definition of $q_i^*$ on infinite trees, one can find a minimal index $j \leq i$ such that $T - v_j$ has more than one infinite component. We will show that $q_j^*(T_n,v^{(n)}) \to 1$ as $n \to \infty,$ which immediately implies $q_i^*(T_n,v^{(n)}) \to 1$ as $q_k^* \leq q_{k+1}^* \leq 1$ for any index $k.$ Let $\epsilon>0$. By the choice of $j$, we can choose a radius $R$ that is so large that all the following conditions are satisfied:
    \begin{enumerate}
        \item all finite branches of $v_j$ lie entirely inside $T|_{v_j,R}$
        \item each infinite branch of $v_j$ restricted to $T|_{v_j,R}$ has more subtrees containing the root than any finite branch, where the root of a branch is the neighbour of $v_j$ in the branch;
        \item $\frac{1}{R+1}<\epsilon$.
    \end{enumerate}
    Now let $R'=R+j$. Since $(T_n,v^{(n)})$ converges to $(T,v)$ in $\mathcal{RT}$, there exists an index $N$ such that $T_n|_{v^{(n)},R'} \cong T|_{v,R'}$ for all $n \geq N$. Now the second condition guarantees that the sequence $v^{(n)}_1,\ldots,v^{(n)}_j$ in $T_n$ coincides with the sequence $v_1,\ldots,v_j$ in $T$ (to be precise, there is an isomorphism from $T_n|_{v^{(n)},R'}$ to $T|_{v,R'}$ that maps $v^{(n)}_k$ to $v_k$ for all $k \leq j$). This further implies that $T_n|_{v^{(n)}_j,R} \cong T|_{v_j,R}$.
    
    Since $v_j$ has at least two infinite branches, the corresponding branches of $v^{(n)}_j$ contain at least $R$ vertices each. Therefore, in the computation of $M(T_n,v^{(n)}_j)$, after removing one branch of $v^{(n)}_j$, there are still at least $R$ subtrees containing $v^{(n)}_j$. Therefore, $M(T_n,v^{(n)}_j) \geq R$ and $1- q_j^*(T_n,v^{(n)}) \leq \frac{1}{R+1}<\epsilon$ for all $n \geq N$. This completes the proof in the first case.
    
\item Next, assume that $q_i^*(T,v)<1$. As $T$ is of infinite order, this means that $v_i$ is well-defined, and $T - v_i$ has only one infinite component. As in the first case, one can find an $R$ that is so large that all finite branches of $v_i$ lie inside $T|_{v_i,R}$ and that the infinite branch of $v_i$ restricted to $T|_{v_i,R}$ has more subtrees containing the root than any finite branch. Therefore, we can again conclude that for large enough $n$, the $R$-ball in $T_n$ centred at $v_i^{(n)}$ is isomorphic to the $R$-ball in $T$ centred at $v_i$. Moreover (still for sufficiently large $n$), in determining $M(T_n,v^{(n)}_i)$, one always removes the branch of $v_i^{(n)}$ that corresponds to the infinite branch of $v_i$ in $T$. These two observations together imply that $M(T_n,v^{(n)}_i)=M(T,v_i)$, thus $q_i^*(T_n,v^{(n)})=q_i^*(T,v)$, for all sufficiently large $n$. \qedhere
\end{enumerate}
\end{proof}

An immediate consequence of Lemma~\ref{lem:cont_qi} is the following. Set $g_K(T_n,v)\coloneqq \prod_{i=0}^K q_i^*(T_n,v)$. Since $q_i^*$ is continuous with values in the interval $(0,1]$ for every fixed $i$ and $g_K$ is a finite product of $q_i^*$s, $g_K$ is also continuous and takes values in $(0,1]$. Now, by the same argument as in the previous section, we have
$$\frac{1}{|T_n|} \sum_{v\in T_n} g_K(T_n,v)=\mathbb{E}_{\sigma_n}(g_K) \to \mathbb{E}_{\sigma}(g_K)=:L_{\sigma,K},$$
where $\sigma_n$ is the uniform measure associated with $T_n$ and $\sigma$ is the limit measure associated with the Benjamini--Schramm limit. In other words, the $\sigma_n$-expectation of $g_K$ converges to the limit $L_{\sigma,K}$, which depends only on $\sigma$ and $K$. Observe that $L_{\sigma,K} = \mathbb{E}_{\sigma}(g_K)$ is bounded and $\mathbb{E}_{\sigma}(g_K)\ge \mathbb{E}_{\sigma}(g_{K+1})$ for all $K \geq 0$. Therefore, $L_{\sigma,K}$ converges, as $K \to \infty$, to a limit $L_\sigma$ that only depends on $\sigma$. We now show that $D(T_n)$ has to converge to this number $L_{\sigma}$, completing the proof of our second main result. 

\begin{proof}[Proof of Theorem~\ref{main local conv theorem}]
We prove that $\mathbb{E}_{\sigma_n}(q)=\frac{1}{|T_n|} \sum_{v\in T_n} q(T_n,v)$ converges to $L_\sigma$. The statement then follows by Lemma \ref{lemma: p and q close enough}. 
    
    Let $\epsilon>0$. By the definition of $L_{\sigma}$ and Lemma \ref{lemma: finite sum good enough}, we can find $K_0$ and $N_0$ such that for $k\ge K_0$ and $n\ge N_0$, we have
    \begin{enumerate}
        \item $\left| L_{\sigma,k}- L_\sigma \right|< \frac{\epsilon}{3},$
        \item $\left| \mathbb{E}_{\sigma_n}(q)- \mathbb{E}_{\sigma_n}(g_k) \right|< \frac{\epsilon}{3}.$ 
    \end{enumerate}
    Now, fix any $k\ge K_0$. As $\mathbb{E}_{\sigma_n}(g_k) \to L_{\sigma,k},$ there exists $N\ge N_0$ such that for all $n\ge N$, we have
    $$
        \left| \mathbb{E}_{\sigma_n}(g_k) - L_{\sigma,k} \right|<\frac{\epsilon}{3}.
    $$
    Putting all three inequalities together, we have
    \begin{align*}
        \left| \mathbb{E}_{\sigma_n}(q)-L_\sigma \right| &\leq \left| \mathbb{E}_{\sigma_n}(q)- \mathbb{E}_{\sigma_n}(g_k) \right| + \left| \mathbb{E}_{\sigma_n}(g_k) - L_{\sigma,k} \right| 
        + \left| L_{\sigma,k}- L_\sigma \right| \\
        &< \frac{\epsilon}{3}+\frac{\epsilon}{3}+\frac{\epsilon}{3} = \epsilon,
    \end{align*}
    for all $n\ge N$. Since $\epsilon$ was arbitrary, the desired statement follows.
\end{proof}

\subsection{Examples}\label{sec:examples2}

We take a look at the examples from Section \ref{sec:number}. As \cref{main local conv theorem} does not apply to the sequence of paths, we start with combs.

\paragraph{Combs.}

Recall that the comb $C_n$ is obtained by attaching one leaf to each of the vertices of a path $P_n$. It was mentioned in Section \ref{sec:examples1}, we know that the number of subtrees of $C_n$ is $N(C_n)=2^{n+2}-n-4$. One way to obtain this formula is by grouping subtrees according to the length of the subpath they induce on the original path. For every $m \leq n$, there are $n-m+1$ such subpaths, each of which can be extended to a subtree in $2^m$ ways. The total number of vertices of all these $2^m$ subtrees is
$$
    m\cdot 2^m+\sum_{i=0}^m i\binom{m}{i}=3m\cdot 2^{m-1}.
$$
Summing over $m$ and adding the $n$ subtrees that only consist of a single leaf yields the following formula for the total number of vertices in all subtrees, which is denoted by $R(C_n)$:
$$
    R(C_n)=n+\sum_{m=1}^n (n-m+1)3m\cdot 2^{m-1} = 3(n-2) 2^{n+1} + 4(n+3), thus
$$
and
$$
    D(C_n)=\frac{\mu(C_n)}{2n}=\frac{\frac{R(C_n)}{N(C_n)}}{2n}=\frac{3(n-2) 2^{n+1} 4(n+3)}{2n(2^{n+2}-n-4)}.
$$
This implies immediately that $D(C_n)\to \frac34$ as $n\to \infty.$

On the other hand, the limit measure $\sigma$ corresponding to the infinite comb $C_\infty$ has probability $\frac12$ for both a leaf-vertex and a path-vertex. If $v$ is a leaf, $M(C_\infty,v)=1$ and $q(C_\infty,v)=\prod_{i\geq1} q^*_i(C_\infty,v)=\frac12,$ otherwise, $q(C_\infty,v)=1$ as $v$ has two branches of infinite vertices. Therefore, $\mathbb{E}_{\sigma}(q)=\frac34,$ in agreement with \cref{main local conv theorem}.

\paragraph{Regular trees (Bethe trees).}

We make use of our earlier considerations in Section~\ref{sec:examples1}. Let us again write $B_k$ for the complete binary tree of height $k$, and note that $RT_n^3$ is obtained by attaching the roots of three copies of $B_{n-1}$ to a common vertex, which is the centre (and also the subtree core) of $RT_n^3$. 

Let the root of $B_k$ be denoted by $v_k$, and recall from the discussion in Section~\ref{sec:examples1} that $m_k = N(B_k,v_k)$ satisfies the recursion $m_k = (m_{k-1}+1)^2$ with $m_0 = 1$. We found that in the infinite canopy tree $RT^3$, which represents the Benjamini--Schramm limit, $M(RT^3,x_k) = m_k$, so we can already compute the value of 
$\mathbb{E}_{\sigma}(q)$. It is
\begin{align*}
    \mathbb{E}_{\sigma}(q) &= \sum_{k \geq 0} 2^{-k-1} \prod_{j\ge 0} q_j^*(RT^3,x_k)\\
    &= \sum_{k \geq 0} 2^{-k-1} \prod_{i\ge k} \frac{M(RT^3,x_i)}{M(RT^3,x_i)+1} \\
    &= \sum_{k \geq 0} 2^{-k-1} \prod_{i\ge k} \frac{m_i}{m_i+1} \\
    &\approx 0.6289684152.
\end{align*}
For comparison, we calculate the limit of the subtree density of $RT_n^3$ as $n \to \infty$. To this end, we need to  determine the total number of vertices in the subtrees of $B_k$ that contain $v_k$, which is denoted by $r_k = R(B_k,v_k)$. By the same bijective correspondence that gave us~\eqref{subtree_recursion}, it is readily seen that $r_n$ satisfies the recursive relation
\begin{equation}\label{eq:rn-rec}
    r_k=2r_{k-1}(m_{k-1}+1)+(m_{k-1}+1)^2.
\end{equation}
Let $w_n$ be the centre of $RT_n^3$. We know from Section~\ref{sec:examples1} that
$$N(RT_n^3,w_n) = (1 + m_{n-1})^3$$
and
\begin{equation}\label{eq:NRT_n}
    N(RT_n^3) = (1+m_{n-1})^3 + \sum_{k=0}^{n-1} 3 \cdot 2^{n-k-1} m_k.
\end{equation}
As mentioned before, subtrees that do not contain $w_n$ are asymptotically negligible due to the rapid growth of the sequence $m_k$. We also have (in the same way as~\eqref{eq:rn-rec})
$$R(RT_n^3,w_n)=3r_{n-1}(m_{n-1}+1)^2+(m_{n-1}+1)^3$$
and thus
$$\mu(RT_n^3) = \frac{R(RT_n^3)}{N(RT_n^3)} \sim \frac{R(RT_n^3,w_n)}{N(RT_n^3,w_n)} = \frac{3r_{n-1}}{m_{n-1}+1} +1$$
Since $\abs{RT_n^3} = 3 \cdot 2^n - 2$, it follows that
$$D(RT_n^3) = \frac{\mu(RT_n^3)}{\abs{RT_n^3}} \sim \frac{r_{n-1}}{2^n m_{n-1}}.$$
Iterating the recursive relations for $m_n$ and $r_n$, we have
\begin{align*}
    \frac{r_{n-1}}{2^n \cdot m_{n-1}} &= \frac{2r_{n-2}(m_{n-2}+1)+(m_{n-2}+1)^2}{2^n \cdot (m_{n-2}+1)^2} \\
    &= \frac{r_{n-2}}{2^{n-1} \cdot (m_{n-2}+1)} + 2^{-n}\\
    &= \frac{r_{n-2}}{2^{n-1} \cdot m_{n-2}} \cdot \frac{m_{n-2}}{m_{n-2}+1} + 2^{-n}\\
    &= \left( \frac{r_{n-3}}{2^{n-2} \cdot m_{n-2}} \cdot \frac{m_{n-3}}{m_{n-3}+1} + 2^{-n+1} \right) \cdot \frac{m_{n-2}}{m_{n-2}+1} + 2^{-n}\\
    &\qquad \vdots \\
    &= 2^{-1}\prod_{i=0}^{n-1} \frac{m_i}{m_i+1} + 2^{-2} \prod_{i=1}^{n-1} \frac{m_i}{m_i+1} + \cdots + 2^{-n+1}\frac{m_{n-2}}{m_{n-2}+1} + 2^{-n}\\
    &= \sum_{k = 0}^{n-2} 2^{-k-1} \prod_{i = k}^{n-2} \frac{m_i}{m_i+1} + 2^{-n}.
\end{align*}
As $n \to \infty$, this tends to
$$\sum_{k \geq 0} 2^{-k-1} \prod_{i\ge k} \frac{m_i}{m_i+1},$$
which is exactly the expression we derived for $\mathbb{E}_{\sigma}(q)$, again in agreement with \cref{main local conv theorem}. 

\section{The set of possible values}\label{sec:values}

Having proved convergence of the quantities $\frac{\log N(T)}{|T|}$ and $D(T)$ under suitable conditions, it is natural to ask what possible values these can take. This question will be answered in our final section. Since the set of series-reduced trees (trees without vertices of degree $2$) often plays a special role in the literature on the subtree density and related quantities, we always provide results for both arbitrary trees and series-reduced trees. Note that a Benjamini--Schramm convergent sequence of series-reduced trees always satisfies the technical condition of Theorem~\ref{main local conv theorem}. In fact, proving Theorem~\ref{main local conv theorem} would be considerably easier for series-reduced trees. A lot of the technicalities are due to the potential presence of long path segments.

Let us now start with the values of the quantity in our first main theorem.

\begin{proposition}\label{prop:ratio_dense}
    The ratio $\frac{\log_2(N(T))}{\abs{T}}$ only takes values in $[0,1]$ for arbitrary trees, and in $[\frac12,1]$ for series-reduced trees. The values are dense in the respective intervals.
\end{proposition}

\begin{proof}
Clearly, $1 \leq N(T) \leq 2^{\abs{T}}$, so the ratio has to lie in the interval $[0,1]$. Now consider a broom consisting of a path with $p$ vertices, to which $k$ leaves are attached at one end. Let this end be $v$. Then we have
\begin{itemize}
\item $2^kp$ subtrees that contain $v$, and
\item $k + \binom{p}{2}$ subtrees that do not contain $v$.
\end{itemize}
For a constant $c \in (0,1)$, set $k = \lfloor cn \rfloor$, $p = n - k$, and call the resulting broom $T_n$. We have
$$\frac{\log_2(N(T_n))}{\abs{T_n}} \sim \frac{k}{n} \to c$$
as $n \to \infty$, showing that the values of the ratio are indeed dense in $[0,1]$.

For series-reduced trees, a simple application of the degree sum formula shows that at least half the vertices are leaves: $L(T) \geq \frac{\abs{T}}{2}$. Considering only those subtrees consisting of all internal vertices and any subset of leaves, we have the trivial lower bound $N(T) \geq 2^{L(T)} \geq 2^{\abs{T}/2}$, so that $\frac12$ becomes a lower bound for the ratio $\frac{\log_2(N(T))}{\abs{T}}$. To show that the values lie dense in the interval $[\frac12,1]$, consider the concatenation of a comb $C_n$, as in our earlier examples, with an additional leaf attached at the left end and a star with $m = \lfloor xn \rfloor$ leaves (for some constant $x$) attached at the right end (connecting the centre of the star to the end of the comb). One can compute the number of subtrees of the resulting tree $T_n'$ to be
$$N(T_n') = 3 \cdot 2^{n+m} + 3 \cdot 2^{n+1} - 2^{m}+ m - n -5 \sim 3 \cdot 2^{n+m},$$
thus $\log_2(N(T_n')) \sim n+m$, while $\abs{T_n'} = 2n+m+2$. Since $\frac{m}{n} \to x$ as $n \to \infty$, it follows that
$$\frac{\log_2(N(T_n'))}{\abs{T_n'}} \to \frac{1+x}{2+x},$$
and this rational function takes all values in $(\frac12,1)$ as $x \in (0,\infty)$.
\end{proof}

Next, we consider the subtree density. As mentioned in the introduction, we have $\frac13 \leq D(T) \leq 1$ for every tree $T$ (as shown by Jamison \cite{jamison1983average}), and the discussion at the beginning of Section~\ref{sec:density} shows that the set of possible values is indeed dense in the interval. For series-reduced trees, the stronger bounds $\frac12 \leq D(T) \leq \frac34$ have been established by Vince and Wang \cite{VW10}. We show that the values of the subtree density are again dense.

\begin{proposition}
    The subtree density $D(T)$ only takes values in $[\frac13,1]$ for arbitrary trees, and in $[\frac12,\frac34]$ for series-reduced trees. The values are dense in the respective intervals.
\end{proposition}

\begin{proof}
We only need to consider series-reduced trees. We can use the same sequence $T_n'$ of trees as in the proof of Proposition~\ref{prop:ratio_dense}. Again, one can compute the total number of vertices in all subtrees exactly: it is given by
$$R(T_n') = 2^{m+n-1}(9n+3m+2) + (9n-13)2^n - 2^{m-1} (m-4) + (4n+m+14) \sim 2^{m+n-1}(9n+3m).$$
Dividing by $\abs{T_n'}N(T_n')$ as given in the proof of \cref{prop:ratio_dense} and plugging in $m = \lfloor xn\rfloor$, we find that
$$\lim_{n \to \infty} D(T_n') = \frac{3+x}{4+2x},$$
and this rational function takes all values in $(\frac12,\frac34)$ as $x \in (0,\infty)$.
\end{proof}

Lastly, we consider the probability $p(T,v)$ that a random subtree contains a given vertex $v$. We have the following result.

\begin{proposition}\label{prop:pTv_dense}
   For arbitrary trees $T$, the values of $p(T,v)$ are dense in the interval $[0,1]$. For series-reduced trees $T$ with $\abs{T} > 2$, $p(T,v)$ has to lie in $[\frac13,\frac12]$ if $v$ is a leaf, and in $[\frac23,1]$ otherwise. The possible values are dense in the respective intervals.
\end{proposition}

\begin{proof}
We can first consider brooms as in the proof of Proposition~\ref{prop:ratio_dense}. Note that the proportion of subtrees that contain the vertex $v$ (where the leaves are attached to the path) tends to $1$, so we only need to consider subtrees containing $v$. The vertex at distance $d$ from $v$ on the path is contained in $\frac{p-d}{p}$ of these subtrees. The values of this quotient are dense in $[0,1]$, so the first statement follows.

For series-reduced trees, we use some results from~\cite{CW26}: letting $\overline{N}(T,v) = N(T) - N(T,v)$, it is shown there that $2N(T,v) - 4 \geq \overline{N}(T,v)$ if $v$ is a leaf, which implies that
$$p(T,v) = \frac{N(T,v)}{N(T)}  = \frac{N(T,v)}{N(T,v) + \overline{N}(T,v)} \geq \frac{N(T,v)}{N(T,v) + 2N(T,v) - 4} \geq \frac13.$$
On the other hand, $p(T,v) \leq \frac12$ holds for every tree $T$ with at least three vertices (series-reduced or not) if $v$ is a leaf: this is because for every subtree that contains $v$ and is not just the single vertex $v$, one can obtain another subtree by removing $v$. This already yields $\overline{N}(T,v) \geq N(T,v) - 1$, and by considering any subtree that contains neither $v$ nor its neighbour, one can strengthen the inequality to $\overline{N}(T,v) \geq N(T,v)$. The bound $p(T,v) \leq \frac12$ follows.

For an internal vertex $v$ of a series-reduced tree, the inequality $N(T,v) -2 \geq 2\overline{N}(T,v)$ is proven in~\cite{CW26}: from this, we obtain
$$p(T,v) = \frac{N(T,v)}{N(T)}  = \frac{N(T,v)}{N(T,v) + \overline{N}(T,v)} \geq \frac{N(T,v)}{N(T,v) + \frac{N(T,v) - 2}{2}} \geq \frac23.$$

It remains to show that the values are dense in the respective intervals. For a value $q \in (\frac23,1)$, set $q=\frac{2}{2+x}$, i.e. $x=\frac2{q} -2$. Note that $x \in(0,1)$. Now $x$ has a binary representation
$$x = 2^{-c_1}+2^{-c_2}+2^{-c_3}+\cdots,$$
where $c_1<c_2<c_3< \ldots$ is an increasing (possibly infinite) sequence of positive integers.

    Let $d_1=c_1$ and $d_i=c_{i}-c_{i-1}$ for $i \ge 2.$
    Given the sequence $(d_1, d_2, \ldots, d_t)$, we construct a caterpillar tree, consisting of a path with $t+1$ vertices $v_1,v_2,\ldots,v_t,v_{t+1}$ to which we attach $d_1+1, d_2, \ldots, d_t$ and $k$ leaves, respectively (see Figure~\ref{fig:f_v}). 
    As we let $k \to \infty$, the proportion of subtrees containing the right end $v_{t+1}$ of the path tends to $1$, so we only need to consider those subtrees. The number of subtrees that contain $v_i,v_{i+1},\ldots,v_{t+1}$, but not $v_{i-1}$, is $2^{d_i+d_{i+1}+\cdots+d_t+k}$ if $i > 1$, and there are $2^{1+d_1+d_2+\cdots+d_t+k}$ subtrees containing all $t+1$ vertices of the path. So we obtain, for the vertex $u = v_1$,
\begin{align*}
p(T,u) &\sim \frac{2^{1 + d_1+d_2+\cdots + d_k}}{2^{1 + d_1+d_2+\cdots + d_k} + \sum_{i=2}^{t+1} 2^{d_i+d_{i+1}+\cdots+d_t+k}} = \frac{2}{2 + \sum_{i=2}^{t+1} 2^{-d_1-d_2-\cdots-d_{i-1}}} \\
&= \frac{2}{2 + \sum_{i=2}^{t+1} 2^{-c_{i-1}}} = \frac{2}{2 + 2^{-c_1} + 2^{-c_2} + \cdots + 2^{-c_t}}.
\end{align*}
Since $t$ is arbitrary, this can be made arbitrarily close to $\frac{2}{2+x} = q$, proving that the values are dense in $[\frac23,1]$ for internal vertices. To show the statement for leaves, simply consider a leaf $w$ that is adjacent to $u$. Since exactly half the subtrees that contain $u$ also contain $w$, we have
$$p(T,w) \sim \frac12 p(T,u) \sim \frac{1}{2 + 2^{-c_1} + 2^{-c_2} + \cdots + 2^{-c_t}},$$
which can be made arbitrarily close to $\frac{q}{2} = \frac{1}{2+x}$. This proves that the set of possible values is dense in $[\frac13,\frac12]$ in this case.
\end{proof}

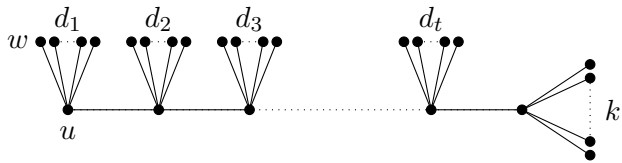
\begin{figure}[ht]
    \centering
      \begin{tikzpicture}[scale=0.6]
        \draw[dotted] (12,1) -- (2,1);

        \draw (2,1)--(6,1);
        \draw (10,1)--(12,1);
        
      \foreach \i in {2,4,6,10}
        {
        \draw[fill=black] (\i,1) circle (3pt);
        \foreach \j in {-0.6,-0.3,0.3,0.6}{
            \draw (\j+\i,2.5) -- (\i,1);
            \draw[fill=black] (\i+\j,2.5) circle (3pt);
        }
            \draw[dotted] (\i-0.3,2.5) -- (\i+0.3,2.5);
        }

        \node (w) at (0.9,2.5) {$w$} ;
        \node (u) at (2,0.5) {$u$} ;
        \node () at (2,3) {$d_1$} ;
        \node () at (4,3) {$d_2$} ;
        \node () at (6,3) {$d_3$} ;
        \node () at (10,3) {$d_t$} ;
        \node () at (14,1) {$k$} ;
        
        \foreach \i in {0,0.3,1.7,2}
        {
            \draw (13.5,\i) -- (12,1);
            \draw[fill=black] (13.5,\i) circle (3pt);
        }
        \draw[dotted] (13.5,0.3) -- (13.5,1.7);
        \draw[fill=black] (12,1) circle (3pt);
    \end{tikzpicture}
    \caption{The construction in the proof of Proposition~\ref{prop:pTv_dense}.}
    \label{fig:f_v}
\end{figure}

\bibliographystyle{abbrv}
\bibliography{bibliography}

\end{document}